\documentclass[12pt,a4paper,reqno]{amsart}

\usepackage[margin=2.4cm]{geometry}
\usepackage{graphics, amsmath,enumitem, comment}
\usepackage{graphicx}
\usepackage{epsfig}

\newif\ifcolorcomments
\newcommand{\allowcomments}[4]{
\newcommand{#1}[1]{\ifdraft{\ifcolorcomments{\textcolor{#4}{##1 --#3}}\else{\textsl{ ##1 \ --#3}}\fi}\else{}\fi}
}

\allowcomments{\commumtaz}{MH}{Mumtaz}{green}
\allowcomments{\comwang}{BW}{BWWang}{blue}
\allowcomments{\comkle}{DK}{DK}{magenta}
\allowcomments{\comnick}{NW}{Nick}{red}

\colorcommentstrue
\usepackage{amssymb}
\usepackage{amsmath,amsthm}
\usepackage{mathrsfs}
\usepackage{bbm}
\def\bc{\begin{center}}
\def\ec{\end{center}}
\def\be{\begin{equation}}
\def\ee{\end{equation}}

\def\N{\mathbb N}

\def\R{\mathbb R}

\def\LL{\mathcal L}

\newtheorem{lem}{Lemma}[section]

\newtheorem{proposition}[lem]{Proposition}
\newtheorem{theorem}[lem]{Theorem}
\newtheorem{lemma}[lem]{Lemma}

\newtheorem{corollary}[lem]{Corollary}

\theoremstyle{remark}
\newtheorem{remark}[lem]{\bf Remark}
\numberwithin{equation}{section}
\usepackage{colortbl}

\newif\ifdraft\drafttrue

\begin{document}

\subjclass[2010]  {}

\title{Dynamical Borel-Cantelli lemma for recurrence theory}

\author[M. Hussain]{Mumtaz Hussain}
\address{Mumtaz Hussain, La Trobe University, POBox199, Bendigo 3552, Australia. }\email{M.Hussain@latrobe.edu.au}
\author[B. Li ]{Bing  Li}
\address {Bing Li, School of mathematics, South China University of Technology, Guangzhou, 510640, China.} \email{scbingli@scut.edu.cn}
\author[D. Simmons]{David Simmons}
\address{David Simmons,  The University of York, England, UK.}  \email{David.Simmons@york.ac.uk}
\author[B-W. Wang]{Baowei Wang}
\address{Bao-wei Wang, School  of  Mathematics  and  Statistics,  Huazhong  University  of Science  and  Technology, 430074 Wuhan,  China.}
 \email{bwei\_wang@hust.edu.cn}

\begin{abstract} We study the dynamical Borel-Cantelli lemma for recurrence sets in a measure preserving dynamical system $(X, \mu, T)$ with a compatible metric $d$. We prove that,   under some regularity conditions,   the $\mu$-measure of the following set
\[
 R(\psi)= \{x\in X : d(T^n x, x) < \psi(n)\ \text{for infinitely many}\ n\in\N \}
\]
obeys a zero-full law according to the convergence or divergence of a certain series, where $\psi:\N\to\R^+$.  Some of the applications of our main theorem include the continued fractions dynamical systems, the beta dynamical systems, and the homogeneous self-similar sets.
\end{abstract}

\maketitle

\section{Introduction}

Poincar\'{e}'s recurrence theorem is one of the most fundamental results in a dynamical system which concerns the properties of the distribution of orbits.
More precisely, let $(X, \mathcal{B}, \mu, T)$ be a measure preserving system with a compatible metric $d$, that is,  $(X, d)$ is a metric space, $\mathcal{B}$ is a Borel $\sigma$-algebra of $X$, and $\mu$ is an $T$-invariant probability measure. If $(X, d)$ has a countable base then Poincar\'{e}'s recurrence theorem states that $\mu$-almost every $x\in X$ is recurrent in the
sense that  $$\liminf_{n\to\infty}d(T^nx, x)=0.$$

However,  this result gives no information about the speed at which a generic orbit $\{T^nx\}_{n\geq 0}$ comes back to the starting point or the shrinking neighbourhood.  A question of great importance is to determine conditions under which the rate of recurrence can be quantified for general dynamical systems.  In particular,  the focus is on the size of the following set:
\begin{equation*}\label{f1}
R(\psi):=\Big\{x\in X: d(T^nx, x)<\psi(n) \ {\text{for i.m.}}\ n\in \N\Big\}
\end{equation*}
 where $\psi:\N\to \R^+$ is a positive function and {\em i.m.} denotes {\em infinitely many}.
%
%

 The most significant and one of the first quantitative recurrence results is due to Boshernitzan \cite{Bo}.
 \begin{theorem}[\cite{Bo}]\label{t1} Let $(X, \mathcal{B}, \mu, T)$ be a measure preserving system with a compatible metric $d$. Assume that the $\alpha$-dimensional Hausdorff measure of $X$ is $\sigma$-finite for some $\alpha>0$. Then for $\mu$-almost
all $x\in X$,
\begin{equation}\label{Bo1}
\liminf_{n\rightarrow \infty }n^{1/\alpha }d(T^{n}x,x)<\infty.\ \
\end{equation}\end{theorem}
In the same paper, in the concluding remarks,  Boshernitzan stated another theorem which deals with a situation when no apriori size of $(X, d)$ was known. 
\begin{theorem}[\cite{Bo}]\label{t2}
Let $(X, \mathcal{B}, \mu, T)$ be a measure preserving system with a compatible metric $d$ such that $(X,d)$ is $\sigma$-compact. Then there exists a sequence $\{a_n\}_{n\ge 1}$ with $a_n\to \infty$ as $n\to\infty$ depending on $(X,d)$, such that
for almost all $x\in X$, $$
\liminf_{n\to\infty}a_n\ d(T^nx, x)=0.
$$
\end{theorem}%

Theorem \ref{t1} was improved by  Barreira-Saussol \cite{BaS} who showed that the exponent $\alpha$  can be replaced by the lower local dimension of a measure at $x$. For piecewise $C^2$ expanding maps with the ergodic measure equivalent to Lebsgue measure, Kirsebom-Kunde-Persson \cite{KKP} improved the speed in \eqref{Bo1} from $n$ to $n(\log n)^{\theta}$ with $\theta<1/2$.

\medskip

As far as a general error function $\psi$ is concerned, hardly anything is known. The only known results for $\mu$-measure of $R(\psi)$ are recently proven by Chang-Wu-Wu \cite{CWW2019}, Baker-Farmer \cite{Baker}, and Kirsebom-Kunde-Persson \cite{KKP}. Chang-Wu-Wu \cite{CWW2019} considered homogeneous self similar set satisfying the strong separation condition. Baker-Farmer \cite{Baker} generalised Chang-Wu-Wu's result to the finite conformal iterated function systems with open set condition. On the other hand, Kirsebom-Kunde-Persson \cite{KKP} presented the recurrence and shrinking target theory when $T$ is an integer matrix action with some condition about the eigenvalues. However, all of  these results are not applicable to some well known dynamical systems, for example,  $\beta$-dynamical systems or the dynamical systems of continued fractions.  We remedy this shortfall in this paper by providing a criterion on the size of $R(\psi)$ applicable to general dynamical systems satisfying certain conditions.



\medskip

Throughout we take $X$ to be a compact subset of $\mathbb{R}^d$. Let $\{X_i\}_{i\in\mathcal{I}}$
be a countable family of non-empty pairwise disjoint subsets of $X$ such that each $X_i$ is open in $X$. Suppose that  $T:~X\to~X$ is Borel measurable and for all $i\in\mathcal{I}$, $T|_{X_i}$ is a $C^1$ map. Furthermore, we assume that $T$ is expanding meaning that $\|(D_xT)^{-1}\|^{-1}>1$ for any $x\in X$. By the notation $D_xT$ we mean  the derivative of $T$  at a point $x\in X$ and $\| D_xT \|=\sup\limits_{v\in X}\frac{\|D_xT(v)\|_2}{\|v\|_2}$.
Let $\mu$ be a $T$-invariant probability measure and $$\mu\left(X\setminus\cup_{i\in\mathcal{I}}X_i\right)=0.$$

We will make consistent use of the following conditions.

 \medskip

\noindent{\bf Condition I (Ahlfors Regular):} the measure  $\mu$ is Ahlfors regular of dimension $\delta>0$, that is, there exist positive constants $\eta_1, \eta_2$ such that for any ball $B(x,r)\subset X$ with $x\in X$,
\begin{equation*}\label{1}\eta_1r^\delta\leq\mu(B(x,r))\leq \eta_2 r^\delta.\end{equation*}

 \medskip

\noindent{\bf Condition II (Exponentially Mixing):}  there
exist constants $C > 0$ and $0~<~\gamma <~1$ such that for any ball $E\subset X$ and measurable set $F\subset X$,
\begin{equation*}\label{2}|\mu(E\cap T^{-n}F)-\mu(E)\mu(F)|\leq C\gamma^n\mu(F), \ \ {\text{for all}}\ n\ge 1.\end{equation*}

 \medskip

\noindent{\bf Condition III (Bounded Distortion):} there exists $K>0$ such that
\begin{equation*}\label{3}K^{-1}\leq \frac{\|D_x(T^n)\|}{\|D_y(T^n)\|}\leq K\end{equation*}
for any $n\in\mathbb{N}$ and $x, y$ in a same cylinder $J_n\in\mathcal{F}_n$. Here $\mathcal{F}_n$ denotes the collection of cylinders of order $n$, that is,
$$\mathcal{F}_n:=\{X_{i_0}\cap T^{-1}X_{i_1}\cap\cdots\cap T^{-(n-1)}X_{i_{n-1}}: i_0, i_1, \dots, i_{n-1}\in\mathcal{I}\}.$$

\noindent{\bf Condition IV:} Denote
$K_{J_n}:=\inf\limits_{x\in J_n}\|D_x(T^n)\|.$ Assume that there exists a universal constant $K>0$ such that \begin{equation*}\label{f6}
\sum_{J_n\in \mathcal{F}_n}\Big(K_{J_n}\Big)^{-\delta}\le K, \  {\text{for all}}\ n\in \N.
\end{equation*}

\noindent{\bf Condition V (Conformality):}
 There exists a constant $C>0$ such that
 for any $J_n\in\mathcal{F}_n$ and ball $B(x_0, r)\subset J_n$, \begin{equation*}\label{4}B(T^n(x_0), C^{-1}K_{J_n} r)\subset T^nB(x_0,r)\subset B(T^n(x_0), CK_{J_n} r).\end{equation*}

Our main result is the following.

%
%
%

\begin{theorem}\label{maintheorem}
Let $\mu$ be a probability measure and $\psi$ be a positive function on $\mathbb{N}$. Suppose that $\mu$ satisfies the conditions \textup{(I--V)}.  Then
\begin{align*}
\mu(R(\psi)) = \left\{\begin{array}{cl}
0& {\rm if} \quad  \sum_{n=1}^\infty \psi^\delta(n)<\infty,\\[2ex]
1& {\rm if} \quad  \sum_{n=1}^\infty \psi^\delta(n)=\infty.
 \end{array}\right.
\end{align*}
\end{theorem}



\medskip

An immediate consequence of this theorem is the following strengthening of the Boshernitzan's results.

\begin{corollary}
Under the setting above, for $\mu$-a.e. $x\in X$,
$$\liminf_{n\to\infty}\psi^{-1}(n)d(T^nx, x)=0,\ {\text{or}}\ \infty,$$ if
$\sum_{n=1}^\infty \psi^\delta(n)<\infty, \ {\text{or}} \ =\infty,$ respectively.
\end{corollary}


We give some remarks in comparing our result with the results of Boshernitzan, Chang-Wu-Wu and Baker-Farmer.  Recall the recurrence set$$
R(\psi)=\Big\{x\in X: d(T^nx, x)<\psi(n) \ {\text{for i.m.}}\ n\in \N\Big\}.
$$

\begin{remark}[Boshernitzan \cite{Bo}] \

\begin{itemize}\item Theorems \ref{t1} and \ref{t2} provide a convergence speed $\psi$ such that $d(T^nx, x) \to~0$, but we do not know whether the convergence speed is optimal.

\item The general convergence speed in Theorem \ref{t2} depends on the underlying dynamical system $(X,T)$.

\item Boshernitzan also remarked, in the concluding remarks, that there is no `universal' convergence speed suitable for {\em all} dynamical system.
So this indicates, more or less, that if a universal function is wanted, the system must satisfy some additional conditions.
\end{itemize}
\end{remark}

\begin{remark}[Chang-Wu-Wu \cite{CWW2019}, Baker-Farmer \cite{Baker}] The results of Chang-Wu-Wu \cite{CWW2019} and Baker-Farmer \cite{Baker} are applicable to finite conformal iterated function systems with open set condition and the map $T:X\to X$ induced by the left shift. Generally speaking the set $$
   A_n:=\Big\{x\in X: d(T^nx, x)<\psi(n)\Big\}
    $$ concerns the distribution of the periodic points $$P_n:=\{x\in X: T^nx=x\}.$$ For finite conformal iterated function system with open set condition, it corresponds to a finite full shift symbolic space making it convenient to study the corresponding set. 
    \begin{itemize}\item
    The set $P_n$ can be precisely expressed as the points in $P_n$ are almost uniformly distributed in $X$.

    \item The finiteness of the iterated function system ensures that a ball is equivalent to a cylinder set. So everything can be translated to a finite full shift symbolic space.

        \item The natural measure supported on $X$ is a Gibbs measure, so it has a nice Bernoulli property which leads to quasi-independence of the sets in question.\end{itemize}

\end{remark}


\begin{remark}[Chernov-Kleinbock, \cite{CK2001}] We take this opportunity to compare the recurrence set above with the shrinking target set. Define the shrinking target set $$
S(\psi)=\Big\{x\in X: d(T^nx, x_0)<\psi(n), \ {\text{i.m.}}\ n\in \N\Big\}=\limsup_{n\to\infty}T^{-n}B(x_0, \psi(n))
$$ for some fixed $x_0\in X$. A dynamical Borel-Cantelli lemma for this setting was presented by Chernov-Kleinbock \cite{CK2001}. For this set,
\begin{itemize}
  \item Since the $\mu$ measure is invariant,  the measure of events $B_n:=T^{-n}B(x_0, \psi(n))$ can be calculated easily.

  \item The mixing property, {\bf Condition (II)}, together with the invariance property of $\mu$, can be applied directly to verify the quasi-independence of the events $\{B_n\}_{n\ge1}$.
\end{itemize}
\end{remark}
\begin{remark}[Our method] In our setting, the set $P_n$ cannot be constructed easily.  The events in our setting cannot be expressed as the $T$-inverse image of some sets, so the invariance of $\mu$ and the mixing property cannot be used directly. The way to overcome these difficulties is to look at the set
$A_n$ locally, then locally $A_n$ behaves like $T^{-n}B_n$ for some $B_n$. Then the invariance and the mixing property of $\mu$ can be applied. It should be noted that this will lead to a superposition of the error terms, which makes the problem more involved. 


\end{remark}

%

\medskip

\noindent {\bf Acknowledgements.}  M.H. was supported by the Australian Research Council Discovery Project (200100994).  B.L. was supported partially by NSFC 11671151 and Guangdong Natural Science Foundation 2018B0303110005.  D.S. was supported by the Royal Society Fellowship. B.W. was supported by NSFC 11722105.  Part of this work was carried out when B.L. and D.S. visited La Trobe University. Thanks to La Trobe University and MATRIX research institute for travel support.

\section{Proof of Theorem \ref{maintheorem}}

We split the proof of the theorem into several subsections for convenience. Let
$$A_n:=\{x\in X : d(T^nx,x) < \psi(n)\},$$
then $R(\psi)=\limsup\limits_{n\to\infty} A_n.$

For notational simplicity, we use $a\lesssim b$ or $a=O(b)$ to say $a\le Cb$ for some unspecified constant $C>0$; and $a\asymp b$ if $a\lesssim b$ and $b\lesssim a$.


 \subsection{The measure of $A_n$.}

In this subsection, the main aim is to prove the following proposition.
\begin{proposition}\label{prop1}
Assume that conditions (I) and (II) hold. Then
\begin{equation*}\label{equivalentseries}
 \sum_{n=1}^\infty \psi^\delta(n)=\infty\Longleftrightarrow \sum_{n=1}^\infty\mu(A_n)=\infty.
 \end{equation*}
\end{proposition}


As previously stated, the set $A_n$ cannot be expressed in the form $T^{-n}B_n$ for some $B_n$. However,  if considered locally, it can be expressed in this form.
\begin{lemma}
\label{lemmaAnapprox}
 Let $B = B(x_0, r)$ be a ball centred at $x_0\in X$ and radius $r>0$.  Then for any  $n\in\mathbb{N}$ with $\psi(n)>r$ and any subset $E$ of $B$,
\[
E \cap T^{-n}\Big(B(x_0,\psi(n)-r)\Big) \subset E\cap A_n \subset E \cap T^{-n}\Big(B(x_0,\psi(n)+r)\Big).
\]
\end{lemma}

\begin{proof}
Fix  a point $x\in E\cap A_n$, then $d(x, x_0)<r$ and $d(T^n x, x) <~\psi(n)$. By using the triangle inequality
\begin{align*}d(T^nx, x_0)&\leq d(T^nx, x)+d(x, x_0)<\psi(n)+r.\end{align*}
That is, $x\in T^{-n}(B(x_0,\psi(n)+r))$. Therefore, $$E\cap A_n \subset E \cap T^{-n}\Big(B(x_0,\psi(n)+r)\Big).$$
The left hand side inclusion follows similarly.
\end{proof}
\begin{remark}
The lemma above gives us a way to write the set $A_n$ as the inverse of a ball with a fixed center by restricting it to a smaller ball. If we choose the ball $B= B(x_0,\epsilon\psi(n))$ with $0<\varepsilon< 1$, then the above lemma yields
\[
B \cap T^{-n}(B(x_0,(1-\epsilon)\psi(n) )) \subset B\cap A_n \subset B \cap T^{-n}(B(x_0, (1+\epsilon)\psi(n))).
\]
\end{remark}

For any ball $B$, with the Lemma \ref{lemmaAnapprox} at our disposal, we are in a position to estimate the measure of $B\cap A_n$.

\begin{lemma}\label{estimateAn}
Let $0<\varepsilon\le \frac{1}{2}$  and $B = B(x_0,\epsilon\psi(n))$ with fixed $x_0\in X$. Assume that conditions (I) and (II) hold.  Then
\begin{align}\label{intervalAn}
\mu(B\cap A_n) &\geq
C_1\mu(B) \psi^\delta(n) - C_2\gamma^n \psi^\delta(n) \\
\mu(B\cap A_n) &\leq C_3\mu(B) \psi^\delta(n) + C_3\gamma^n \psi^\delta(n)\label{ine8},
\end{align}
where $C_1=\eta_1 (1-\varepsilon)^\delta, C_2=\eta_2 C(1-\varepsilon)^\delta, C_3=\max\{\eta_2(1+\varepsilon)^\delta,  \eta_2C(1+\varepsilon)^\delta\}$ are constants and
 $\eta_1, \eta_2, C$ are constants arising from conditions (I) and (II).
\end{lemma}
\begin{proof} We prove inequality (\ref{intervalAn}) only as the proof of inequality \eqref{ine8} follows similarly.

Using the left inclusion in Lemma \ref{lemmaAnapprox} and then the mixing property of $\mu$ ({\bf Condition II}), we have\begin{align*}
\mu(B\cap A_n)
&\geq \mu\Big(B\cap T^{-n} B(x_0,(1-\varepsilon)\psi(n))\Big)\\
&\geq \mu(B)\cdot \mu\Big(B(x_0,(1-\varepsilon)\psi(n))\Big) - C\gamma^n \mu\Big(B(x_0,(1-\varepsilon)\psi(n))\Big).\end{align*} Now using the Ahlfors regularity of $\mu$ ({\bf Condition I}), we conclude that
\begin{align*}\mu(B\cap A_n)\geq \eta_1(1-\varepsilon)^\delta \mu(B) \psi^\delta(n) - \eta_2 C(1-\varepsilon)^\delta\gamma^n \psi^\delta(n).
\end{align*}
\end{proof}
The next lemma estimates the $\mu$-measure for the set $A_n$.

\begin{lemma}\label{estimatemeasures}
Let $0<\varepsilon<1$ and $n\in\mathbb{N}$. Assume that conditions (I) and (II) hold. Then
$$C_4\psi^\delta(n)-C_5\gamma^n\varepsilon^{-\delta}\leq\mu(A_n)\leq C_6\psi^\delta(n)+C_6\gamma^n\varepsilon^{-\delta},$$
where $C_4= \eta_1\eta_2^{-1}5^{-\delta}C_1, \ C_5=5^{-\delta}\varepsilon^{-\delta}C_2,$ and $C_6=\max\{\eta_1^{-1}\eta_2C_3 5^\delta, \eta_1^{-1}\varepsilon^{-\delta}C_3\}$ are constants.
\end{lemma}
\begin{proof}
Consider the collection of balls $$\Big\{B(x, \varepsilon\psi(n)): x\in X\Big\},$$ which naturally covers $X$. By Vitali's covering theorem (or commonly known as $5r$ covering lemma), we can find countably many disjoint balls $\{B(x_j, \varepsilon\psi(n))\}_{j\in \mathcal{J}}$ such that
\begin{equation}\label{covering}
\bigcup_{j\in \mathcal{J}}B(x_j, \varepsilon\psi(n))\subset X\subset \bigcup_{j\in \mathcal{J}}B(x_j, 5\varepsilon\psi(n)).
\end{equation}
 By the left inclusion of (\ref{covering}) and the disjointness of $\{B(x_j, \varepsilon\psi(n))\}_{j\in \mathcal{J}}$, we have
\begin{align*}
\sum_{j\in \mathcal{J}}\eta_1(\varepsilon\psi(n))^\delta&\leq\sum_{j\in \mathcal{J}}\mu(B(x_j, \varepsilon\psi(n))) \\ &=\mu\left(\bigcup_{j\in \mathcal{J}}B(x_j, \varepsilon\psi(n))\right)\\&\leq \mu(X)= 1.\end{align*}
So the cardinality $\mathcal{N}$ of $\mathcal{J}$ is bounded from above by $\eta_1^{-1}(\epsilon\psi(n))^{-\delta}$. Similarly, by the right inclusion of \eqref{covering}, we have
\begin{align*}1=\mu(X)&= \mu(\bigcup_{j\in \mathcal{J}}B(x_j, 5\varepsilon\psi(n)))\\& \leq\sum_{j\in \mathcal{J}}\mu(B(x_j,5 \varepsilon\psi(n))) \\ &\leq \sum_{j\in \mathcal{J}}\eta_25^{\delta}(\varepsilon\psi(n))^\delta.\end{align*}
Thus $\mathcal{N}$ is bounded from below by $\eta_2^{-1}5^{-\delta}(\epsilon\psi(n))^{-\delta}$.

It is clear that \begin{equation}\label{f4}A_n\subset\bigcup_{j\in \mathcal{J}}(B(x_j, 5\varepsilon\psi(n))\cap A_n).\end{equation} Thus by Lemma \ref{estimateAn},
\begin{align*}
\mu(A_n)&\leq \sum_{j\in \mathcal{J}}\mu\Big(B(x_j, 5\varepsilon\psi(n))\cap A_n\Big)\\ &\leq \mathcal{N}\cdot \bigg[C_3\mu\Big(B(x_j, 5\varepsilon\psi(n))\Big) \psi^\delta(n) + C_3\gamma^n \psi^\delta(n)\bigg]\\
&\leq \eta_1^{-1}\eta_2C_3 5^\delta\psi^\delta(n) + \eta_1^{-1}\varepsilon^{-\delta}C_3\gamma^n.
\end{align*}
The other inequality concerning $\mu$ to be proved can be done by replacing (\ref{f4}) by $$A_n\supset\bigcup_{j=1}^{\mathcal{N}}(B(x_j, \varepsilon\psi(n))\cap A_n).$$
\end{proof}

\begin{proof}[Proof of Proposition \ref{prop1}] Take $\varepsilon=\frac{1}{2}$. Then in view of Lemma \ref{estimatemeasures}  we have that
$$\sum_{n=1}^\infty\mu(A_n)\asymp  \sum_{n=1}^\infty \psi^\delta(n)+ \sum_{n=1}^\infty \gamma^n.$$
Since $0<\gamma<1$, the second term on the right converges and the proof of the proposition is complete.
\end{proof}

\subsection{Estimating the measure of $A_m\cap A_n$ with $m<n$} Recall that $\mathcal{F}_m$ denotes the collection of cylinders of order $m$,
$$\mathcal{F}_m:=\{X_{i_0}\cap T^{-1}X_{i_1}\cap\cdots\cap T^{-(m-1)}X_{i_{m-1}}: i_0, i_1, \dots, i_{m-1}\in\mathcal{I}\}.$$

\begin{lemma}\label{l1}
Let $J_m$ be a cylinder in $\mathcal{F}_m$.  For any open set $U\subset J_m$, we have
$$\mu(T^mU)\asymp K_{J_m}^\delta\mu(U).$$
\end{lemma}
\begin{proof}
For a ball $B(x_0, r)\subset J_m$,  {\bf Condition V} implies that $$
B(T^mx_0, C^{-1}K_{J_m}r)\subset T^mB(x_0,r)\subset B(T^mx_0, CK_{J_m}r).
$$ Then the Ahlfors regularity of $\mu$ implies that
$$\mu(T^mB(x_0,r))\asymp K_{J_m}^\delta\mu(B(x_0, r)).$$
Together with the fact that every open set can be written as the disjoint union of at most countably many balls, the desired result follows.
\end{proof}

\begin{lemma}
Let $J_m$ be a cylinder in $\mathcal{F}_m$. Then
$$\text{rad}(J_m)\lesssim K_{J_m}^{-1}\ \ \text{and}\ \ \mu(J_m)\lesssim K_{J_m}^{-\delta}.$$
\end{lemma}
\begin{proof}
The proof follows straightaway from the expanding rate of $T^m|_{J_m}$ and then the Ahlfors regularity of $\mu$.
\end{proof}
\begin{lemma}\label{restriction}
Let $J_m$ be a cylinder in $\mathcal{F}_m$.  Then there is a ball of radius $r~=~K_{J_m}^{-1}\psi(m)$, say $B(z,r)$,
such that $$J_m\cap A_m\subset B(z,r)\cap J_m:=J_m^*.$$
\end{lemma}
\begin{proof}
Choose $z\in J_m \cap A_m$. For any $x\in J_m\cap A_m$, on the one hand we have $$
d(T^mx,T^mz)\asymp  \|D_{z}(T^m)\|\cdot d(x, z);
$$ and on the other hand, $$
d(T^mx, T^mz)\le d(T^mx, x)+d(x, z)+d(z, T^mz)<2\psi(m)+d(x, z).
$$
Since $T$ is expanding, $$\|D_x(T^m)\| \geq \|(D_x(T^m))^{-1}\|^{-1}\gtrsim1,$$ thus $$
d(x, z)\lesssim \|D_{z}(T^m)\|^{-1}\psi(m).
$$

\end{proof}

\begin{proposition}\label{secondmomentestimate}
Let $m<n$. Then
\begin{eqnarray*}
\mu(A_m\cap A_n)\lesssim \psi^\delta(m)\psi^\delta(n)+\gamma^{n-m} \psi^\delta(n)+O(\gamma^n)\psi^\delta(m).
\end{eqnarray*}
\end{proposition}
\begin{proof}

Write
$$A_m=\bigsqcup_{J_m\in\mathcal{F}_m}J_m\cap A_m\subset\bigsqcup_{J_m\in\mathcal{F}_m}J_m^*.$$
Now we estimate $\mu(J_m^*\cap A_n)$ for any fixed $J_m\in\mathcal{F}_m$. Take $r=K_{J_m}^{-1}\psi(m)$  and the ball $B(z,r)$ as in Lemma \ref{restriction}.
There are two cases.

\subsection*{Case (i):  $r\leq \psi(n)$} 
Applying Lemma \ref{lemmaAnapprox} to $J_m^*$, we have
\begin{eqnarray}\label{f7}
J_m^*\cap A_n\subset J_m^* \cap T^{-n}(B(z,2\psi(n))).\end{eqnarray}
Applying Lemma \ref{l1} to the right hand side of the inequality (\ref{f7}), we have \begin{eqnarray*}
\mu(J_m^*\cap A_n)\lesssim K_{J_m}^{-\delta}\cdot \mu\left(T^m(J_m^*)\cap T^{-(n-m)}\left(B(z,2\psi(n))\right)\right).\end{eqnarray*}
Note that by the conformality of $T$ ({\bf Condition V}), $$
T^mJ_m^*\subset T^mB(z,r)\subset B(T^mz, C \psi(m)).
$$
Finally by the mixing property of $\mu$ (Condition II), it follows that
\begin{align*}
\mu(J_m^*\cap A_n)&\leq K_{J_m}^{-\delta}\bigg[\mu\Big(B(T^mz, C\psi(m))\Big) \cdot \mu\Big(B(z,2\psi(n))\Big) + C\gamma^{n-m} \mu\Big(B(z, 2\psi(n))\Big)\bigg] \\
&\lesssim K_{J_m}^{-\delta}\bigg[\psi^\delta(m) \cdot\psi^\delta(n) + \gamma^{n-m} \psi^\delta(n)\bigg].
\end{align*}
So
\begin{align*}
I_1&=\sum_{J_m\in\mathcal{F}_m \atop r\leq \psi(n)}\mu(J_m^*\cap A_n)\lesssim
\psi^\delta(m) \cdot\psi^\delta(n) + \gamma^{n-m} \psi^\delta(n),\end{align*}
where we have used the boundedness ({\bf Condition IV}) of $\sum_{J_m\in\mathcal{F}_m}K_{J_m}^{-\delta}$.
Finally, by Lemma \ref{estimatemeasures} and taking $\varepsilon=\frac{1}{2}$, we conclude that
\begin{align*}
I_1&\lesssim \Big(\mu(A_m)+O(\gamma^m)\Big)\psi^\delta(n) +\gamma^{n-m} \psi^\delta(n).
\end{align*}

\medskip

\subsection*{Case (ii) $r> \psi(n)$}  We replace the ball $B(z,r)$ by a collection of balls of radius $\psi(n)$. To achieve this, choose a maximal $\psi(n)$-separated points in $B(z,r)$, denoted by $\{z_i\}_{1\le i\le p_{m,n}}$. Then it is clear that
\begin{equation*}\label{f5}
B(z,r)\subset \bigcup_{i=1}^{p_{m,n}} B(z_i, \psi(n)), \ {\text{and}}\ \bigcup_{i=1}^{p_{m,n}} B(z_i, \psi(n))\subset B(z,2r).
\end{equation*}
By the Ahlfors regularity of $\mu$, a volume argument implies that $$
p_{m,n}\asymp \left(\frac{r}{\psi(n)}\right)^{\delta}\asymp \left(\frac{K_{J_m}^{-1}\psi(m)}{\psi(n)}\right)^{\delta}.
$$

Now for each ball $B(z_i, \psi(n))$ with $1\leq i\leq p_{m,n}$, we have
\begin{eqnarray*}
\mu\Big(B(z_i, \psi(n))\cap A_n\Big)&\leq& \mu\bigg(B\Big(z_i, \psi(n)\Big)\cap T^{-n}B\Big(z_i,2\psi(n)\Big)\bigg)\\
&\leq & \Big[\mu\Big(B(z_i, \psi(n))\Big)+O(\gamma^n)\Big]\mu\Big(B(z_i, 2\psi(n))\Big)\\
&\lesssim& \Big[\psi^{\delta}(n)+O(\gamma^n)\Big]\psi^\delta(n).
\end{eqnarray*}
Finally, summing over all $1\leq i\leq p_{m,n}$, we have
\begin{align*}\mu(J_m^*\cap A_n)&\leq\sum_{i=1}^{p_{m,n}}\mu\Big(B(z_i, \psi(n))\cap A_n\Big)\lesssim \Big[\psi^\delta(n)+O(\gamma^n)\Big]\psi^\delta(m)K_{J_m}^{-\delta}.\end{align*}

Therefore,
\begin{eqnarray*}
I_2:=\sum_{J_m\in\mathcal{F}_m \atop r> \psi(n)}\mu(J_m^*\cap A_n)&\lesssim&
\Big[\psi^\delta(n)+O(\gamma^n)\Big]\psi^\delta(m)\cdot \sum_{J_m\in \mathcal{F}_m}K_{J_m}^{-\delta}\\
&\lesssim& \psi^\delta(m)\psi^\delta(n)+O(\gamma^n)\psi^\delta(m).
\end{eqnarray*}

Hence,
\begin{eqnarray*}
\mu(A_m\cap A_n)&=&\sum_{J_m\in\mathcal{F}_m}\mu(J_m^*\cap A_n) =I_1+I_2\\
&\lesssim& \psi^\delta(m)\psi^\delta(n)+\gamma^{n-m} \psi^\delta(n)+O(\gamma^n)\psi^\delta(m).
\end{eqnarray*}
\end{proof}

\subsection{Completing the proof of Theorem \ref{maintheorem}} There are two parts of the proof: the convergence part and the divergence part. The convergence part, however, is a straightforward application of the first Borel-Cantelli lemma and Proposition \ref{prop1} by noting that $$ \sum_{n=1}^\infty \psi^\delta(n)<\infty\Longrightarrow \sum_{n=1}^\infty\mu(A_n)<\infty.$$



The main ingredient in proving the divergence part is the  usage of well-known Paley-Zigmund inequality which enables us conclude the positiveness of $\mu(\limsup A_n)$. Then by a technical way, we conclude the full measure property.

 \subsubsection{Positive measure}

 Let $N\in\mathbb{N}$ and $Z_N(x)=\sum_{n=1}^N\chi_{A_n}(x)$, where $\chi$ is the characteristic function. We first estimate the lower bound for the first moment and then the upper bound for the second moment of the random variable $Z_N$.

\begin{itemize}\item The first moment. By Lemma \ref{estimatemeasures} and choose $\varepsilon=\frac{1}{2}$, for $N$ sufficiently large, one has
\begin{align*}
\mathbb{E}(Z_N)&=\sum_{n=1}^n\mu(A_n)\ge \sum_{n=1}^N\Big(C_4\psi^\delta(n)-C_5\gamma^n\varepsilon^{-\delta}\Big)\\
&\ge C_4\sum_{n=1}^N\psi(n)^{\delta}-C_5'\ge \frac{C_4}{2}\sum_{n=1}^N\psi(n)^{\delta}
\end{align*} where for the second inequality, we used the divergence of $\sum_{n\ge 1}\psi(n)^{\delta}$.

\item The second moment.
\begin{align*}\mathbb{E}(Z_N^2)&=\mathbb{E}\left(\sum_{n=1}^N\chi_{A_n}+2\sum_{1\leq m<n\leq N}\chi_{A_m}\chi_{A_n}\right)\\ &=\sum_{n=1}^N\mu(A_n)+2\sum_{1\leq m< n\leq N}\mu(A_m\cap A_n).\end{align*}

Summing over $m,n$ $(1\leq m<n\leq N)$ in Proposition \ref{secondmomentestimate} gives

\[
\sum_{1\leq m<n\leq N} \mu(A_m\cap A_n) \lesssim \left(\sum_{1\leq n\leq N} \psi^\delta(n)\right)^2 + \sum_{1\leq n\leq N}\psi^\delta(n).\]
Therefore,
\begin{align*}\mathbb{E}(Z_N^2)&=\sum_{n=1}^N\mu(A_n)+\sum_{1\leq m< n\leq N}\mu(A_m\cap A_n)
\\
&\leq  C\left(\sum_{1\leq n\leq N} \psi^\delta(n)\right)^2 + (1+C)\sum_{1\leq n\leq N} \psi^\delta(n).\end{align*}
\end{itemize}

By the Paley-Zygmund inequality, for any $\lambda>0$, we obtain
\begin{align*}\mu\Big(Z_N>\lambda\mathbb{E}(Z_N)\Big)&\geq (1-\lambda)^2\frac{\mathbb{E}(Z_N)^2}{\mathbb{E}(Z_N^2)}\\
&\geq (1-\lambda)^2\frac{\left(\sum_{1\leq n\leq N}\eta_1\psi^\delta(n)\right)^2}{C\left(\sum_{1\leq n\leq N} \psi^\delta(n)\right)^2 + (1+C)\sum_{1\leq n\leq N} \psi^\delta(n)}.\end{align*}
Letting $N\to\infty$ 
we get $$\mu\Big(\limsup A_n\Big) \geq \mu\Big(\limsup(Z_N>\lambda\mathbb{E}(Z_N))\Big)\geq \limsup\mu\Big(Z_N>\lambda\mathbb{E}(Z_N)\Big)>0.$$

%
%
%
%
%

\subsubsection{Full measure}
Consider a subset of $X$:
$$R'(\psi)=\{x\in X: \liminf\limits_{n\to\infty}\psi^{-1}(n)|T^nx-x|<\infty\}.$$
We check that the set $R'(\psi)$ is invariant in the sense that \begin{align}\label{5}
\mu\Big(R'(\psi)\setminus T^{-1}R'(\psi)\Big)=0.
\end{align} More precisely, take a point $x\in R'(\psi)\cap (\cup_{i\ge 1}X_i)$. Let $i\ge 1$, $c(x)>0$ and $\{n_k\}_{k\ge 1}\subset \N$ be such that $$x\in X_i, \ {\text{and}}\ |T^{n_k}x-x|<c(x)\cdot \psi(n_k), \ {\text{for all}}\ k\ge 1.$$ Since $X_i$ is open, then for all $k$ large, $T^{n_k}x\in J_i$ two. So, for each $k\ge 1$ large
\begin{align*}
|T^{n_k}(Tx)-Tx|&=|T(T^{n_k}x)-T(x)| \\ &\le K \|D_x(T)\|\cdot |T^{n_k}x-x| <\tilde{c}(x)\cdot \psi(n_k).
\end{align*} This means that $$
R'(\psi)\cap (\cup_{i\ge 1}X_i) \subset T^{-1}R'(\psi)
$$ which proves (\ref{5}) since $$
\mu(X\setminus \cup_{i\ge 1}X_i)=0.
$$

It is clear that $R(\psi)\subset R'(\psi)$. The exponential mixing property ({\bf Condition II}) implies that $T$ is ergodic. Thus, together with the invariance of $R'(\psi)$,
we have shown that \begin{equation}\label{f8}\sum_{n\ge 1}\psi(n)^{\delta}=\infty\Longrightarrow \mu(R(\psi))>0\Longrightarrow \mu(R'(\psi))>0\Longrightarrow \mu(R'(\psi))=1.\end{equation}

Next we show that $\mu(R(\psi))=1$.  Take a sequence of positive numbers $\{\ell(n): n\ge 1\}$ such that $$
\sum_{n=1}^{\infty}\left(\frac{\psi(n)}{\ell(n)}\right)^\delta=\infty, \ \ \lim_{n\to \infty}\ell(n)=\infty.
$$ Applying (\ref{f8}) to $\widetilde{\psi}(n)=\psi(n)/\ell(n)$, we have that for $\mu$-almost all $x\in X$, $$
\liminf_{n\to \infty}\frac{\ell(n)}{\psi(n)}d(T^nx,x)<\infty.
$$ By Egorov's theorem, for any $\epsilon>0$, there exists $M>0$ such that the set $$
R_M=\Big\{x\in X: \frac{\ell(n)}{\psi(n)}d(T^nx, x)<M, \ {\text{for i.m.}}\ n\in \N\}
$$ is of measure at least $1-\epsilon$. It is clear that $$
R_M\subset R(\psi), \ \ {\text{since}}\ \ell(n)>M, \ {\text{for large}}  \ n\in \N.
$$ Since  $\epsilon$ is arbitrary, we conclude that $$
\mu(R(\psi))=1.
$$
%

\section{Applications} 
In this section we present some applications of Theorem \ref{maintheorem}. There may be many more applications but we have restricted ourselves to some well-known examples. In particular, Theorems \ref{betathm} and \ref{CFrecurrence} given below are new and never appeared in the literature before.  These two theorems gives the dichotomy laws for the Lebesgue measures of the recurrence sets in the $\beta$-dynamical systems and the dynamical systems of continued fractions respectively. In contrast, the Dynamical Borel-Cantelli lemma for the shrinking target problems was studied over fifty years ago by Philipp \cite{Phi}, where he considered the dynamics of $N$-adic expansion, $\beta$-expansion, and continued fractions. 

\subsection{$\beta$-dynamical system}
 For a real number $\beta>1$, define the transformation $T_\beta:[0,1]\to[0,1]$ by $$T_\beta: x\mapsto \beta x\bmod 1.$$
 This map generates the $\beta$-dynamical system $([0,1], T_\beta)$. It is well known that $\beta$-expansion is a typical
 example of an expanding non-finite Markov system whose properties are reflected by the orbit of some critical point.
 General $\beta$-expansions have been widely studied in the literature, beginning with the pioneering works of Renyi \cite{Re57},
 Parry \cite{Pa60}, Schmeling \cite{Schme}, and Tan-Wang \cite{TaW} etc. 

For this application we first check that the $\beta$-dynamical system satisfies all the  conditions stated in our framework.

 \begin{enumerate}

 \item Partition: $$
 X_i=\Big(\frac{i-1}{\beta}, \frac{i}{\beta}\Big); \ {\text{and}}\ \ 1\le i\le \lfloor \beta\rfloor, \ \ X_{\lfloor \beta\rfloor+1}=\Big(\frac{\lfloor \beta\rfloor}{\beta}, 1\Big).
 $$

 \item Ahlfors regularity of the measure. Let $\mu$ be the Parry measure which is equivalent to the Lebesgue measure $\LL$ with the density $$
h(x)={\left(\int_{0}^1 \sum_{n: T^n1<x}\frac{1}{\beta^n}dx\right)^{-1} \sum_{n: T^n1<x}}\frac{1}{\beta^n}.
$$


\item Strong mixing property is due to Philipp \cite{Phi}.
\item  Bounded distortion. Restricted to a cylinder $J_n$ of order $n$, $T^n_{\beta}$ is a linear map with slope $\beta^n$.

\item  $$
\sum_{J_n\in \mathcal{F}_n}\Big(K_{J_n}\Big)^{-\delta}=\sum_{J_n\in \mathcal{F}_n}\beta^{-n}=\beta^{-n}\cdot \# \mathcal{F}_n\le \frac{\beta}{\beta-1},
$$
where the inequality follows from the fact that  $\beta^n\leq\# \mathcal{F}_n\leq \frac{\beta^{n+1}}{\beta-1}$, see \cite{Re57}.
\end{enumerate}

 Hence all the conditions in the main theorem are fulfilled for $\beta$-dynamical system.  Thus,  as an application of our theorem, we are able to give a complete Lebesgue measure of the recurrence set
 \begin{equation*}\label{RTbP}R(T_{\beta}, \psi):=\left\{x\in [0,1]:|T_\beta^nx-x|<\psi(n)\ \mbox{ for i.m.}\ n\in \N\right\},\end{equation*}
 in the $\beta$-dynamical system.
\begin{theorem}\label{betathm}
Let $\mu$ be the Parry measure. Then
\begin{equation*}
\mu(R(T_\beta, \psi))=\left\{\begin{array}{cl}
0& {\rm if} \quad  \sum_{n=1}^\infty \psi(n)<\infty,\\[2ex]
1& {\rm if} \quad  \sum_{n=1}^\infty \psi(n)=\infty.
 \end{array}\right.
\end{equation*}
\end{theorem}


%

\begin{remark}\label{remapp1}
 It should be noted that the $\beta$-dynamical system for a general $\beta>1$ is neither a self-similar set nor a finite conformal iterated function system with open set condition. So the results of Baker-Farmer \cite{Baker} and Chang-Wu-Wu \cite{CWW2019} are not applicable to the beta dynamical systems. Their results are not applicable to  the systems of continued fractions or the dynamical systems generated by Gauss maps either as stated below.
\end{remark}



\subsection{Continued fraction dynamical system}
Let $T_G$ be the Gauss map on $[0, 1)$.
It was shown by Philipp \cite{Phi}, the system $([0, 1), T_G)$ is exponentially mixing with respect to the Gauss measure $\mu$ given by $d\mu=dx/(1+x)\log 2$.  Since the Gauss measure $\mu$ is equivalent to the Lebesgue measure ($\mathcal L$), {\bf Condition I} is satisfied with $\delta=1$. For any irrational $x\in [0, 1)$,
\begin{eqnarray}\label{derivative}
q_n^2(x)\leq |(T_G^n(x))'|\leq 4q_n^2(x),
\end{eqnarray}
where $q_n(x)$ is the denominator of the $n$-th convergent of the continued fraction expansion of $x$.
It follows that given any cylinder $I(a_1,a_2,\cdots, a_n)$ with $a_1,\dots,a_n\in\mathbb{N}$, for any $x,y\in I(a_1,a_2,\cdots, a_n)$,
$$\frac{1}{4}\leq\frac{|(T_G^n(x))'|}{|(T_G^n(y))'|}\leq 4.$$
So {\bf Condition III} also holds.

For any $J_n=I(a_1,a_2,\cdots, a_n)\in\mathcal{F}_n$,
\begin{eqnarray}\label{KJn}
q_n^2(a_1,\dots,a_n)\leq K_{J_n}=\inf_{x\in J_n}|(T_G^n(x))'|\leq 4q_n^2(a_1,\dots,a_n).
\end{eqnarray}
Note that
$$\frac{1}{2q_n^2}\leq |I(a_1,a_2,\cdots, a_n)|=\frac{1}{q_n(q_n+q_{n-1})}\leq \frac{1}{q_n^2},$$
we have
$$\sum_{J_n\in\mathcal{F}_n}K_{J_n}^{-1}\leq \sum_{J_n\in\mathcal{F}_n}q_n^{-2}(a_1,\dots,a_n)\leq \sum_{J_n\in\mathcal{F}_n}2|I(a_1,a_2,\cdots, a_n)|\leq 2.$$
That is, {\bf Condition IV} is satisfied. Since $T_G|_{J_n}$ is monotonic and $C^1$, combining \eqref{derivative} and \eqref{KJn} gives that {\bf Condition V} holds with $C=4$.

Define the recurrence set as
$$R(T_G, \psi)=\{x\in [0, 1): |T_G^nx-x|<\psi(n)\ \text{for i.m. }  n\in \N\}.$$

Thus we can apply Theorem \ref{maintheorem} to this set.

\begin{theorem}\label{CFrecurrence}
Let  $\psi$ be a positive function and $T_G$ be the Gauss transformation. Then
\begin{equation*}
\mathcal{L}(R(T_G, \psi))=\left\{\begin{array}{cl}
0& {\rm if} \quad  \sum_{n=1}^\infty \psi(n)<\infty,\\[2ex]
1& {\rm if} \quad  \sum_{n=1}^\infty \psi(n)=\infty.
 \end{array}\right.
\end{equation*}

\end{theorem}


 \subsection{Homogeneous self-similar sets}

Our result is applicable to a range of self-similar sets but here we demonstrate it for the classical middle-third Cantor set $\mathcal{K}$. Let $T_3$ be the $3$-adic transformation on $K$, $\mu$ the Cantor measure restricted on $\mathcal{K}$, $\delta=\log_3 2$. Then all the conditions are fulfilled for Theorem \ref{maintheorem}.  Let
\[R(T_3, \psi)=\left\{x\in K: |T_3^nx-x|<\psi(n) \ {\text{ for i.m.}}
\ n\in \N\right\}.\]
We have
\begin{theorem}
Let  $\psi$ be a positive function. Then
\begin{equation*}
\mu(R(T_3, \psi))=\left\{\begin{array}{cl}
0& {\rm if} \quad  \sum_{n=1}^\infty \psi(n)^\delta<\infty,\\[2ex]
1& {\rm if} \quad  \sum_{n=1}^\infty \psi(n)^\delta=\infty.
 \end{array}\right.
\end{equation*}
\end{theorem}
%


\medskip

\medskip

\def\cprime{$'$} \def\cprime{$'$} \def\cprime{$'$} \def\cprime{$'$}
  \def\cprime{$'$}

\end{document}